\documentclass[10pt,leqno]{amsart}

\usepackage{amsmath,amsthm}
\usepackage{amssymb}
\usepackage[all]{xypic}
\usepackage{amsbsy}
\usepackage{url}
\usepackage{hyperref}  
\usepackage{enumerate}

\theoremstyle{plain}
\newtheorem{thm}[equation]{Theorem}
\newtheorem{prop}[equation]{Proposition}
\newtheorem{lem}[equation]{Lemma}

\theoremstyle{definition}

\newtheorem{rem}[equation]{Remark}
\newtheorem{exam}[equation]{Example}

\def\quickop#1{\expandafter\newcommand\csname #1\endcsname{\operatorname{#1}}}
\quickop{Hom}
\quickop{End}
\quickop{Aut}
\quickop{map}
\quickop{Tor}
\quickop{holim}

\def\FF{\mathbb{F}}
\def\QQ{\mathbb{Q}}
\def\GG{\mathbb{G}}
\def\SS{\mathbb{S}}
\def\WW{{{\mathbb{W}}}}
\def\ZZ{{{\mathbb{Z}}}}

\def\Ext{\mathrm{Ext}}
\def\longr{{{\longrightarrow\ }}}
\def\CC{{{\mathbb{C}}}}
\def\mm{{\mathfrak{m}}}
\def\FF{{{\mathbb{F}}}}
\def\CP{{{\CC\mathrm{P}}}}
\def\Def{{{\mathrm{Def}}}}
\def\defeq{\overset{\mathrm{def}}=}
\def\Pic{{\mathrm{Pic}}}

\def\RC{{\mathfrak{Z_n}}}
\def\wcom{{ \lim _k\ \ZZ/p^k(p^n-1)}}
\def\pico{{\Pic_{K(n)}^0}}

\def\Sq{{\mathrm{Sq}}}
\def\Gal{{\mathrm{Gal}}}

\date{}
\begin{document}
\title{Comparing Dualities in the $K(n)$-local Category}
  
\author[P. G. Goerss]{Paul G. Goerss}
\address{Department of Mathematics\\ Northwestern University}

\author[M.J. Hopkins]{Michael J. Hopkins}
\address{Department of Mathematics\\Harvard University}

\thanks{Support provided by the National Science Foundation under grant No.~DMS--1810917
and by the Isaac Newton Institute for Mathematical Sciences. The first author was in residence at the Newton
Institute for the program ``Homotopy Harnessing Higher Structures'' supported by the EPSRC grant no EP/K032208/1.
}

\begin{abstract} In their work on the period map and the dualizing sheaf for Lubin-Tate space, Gross and the
second author
wrote down an equivalence between the Spanier-Whitehead and Brown-Comenetz duals of certain
type $n$-complexes in the $K(n)$-local category at large primes. In the culture of the time, these results
were accessible to educated readers, but this seems no longer to be the case; therefore, in this note we give the
details.  Because we are at large primes, the key result is algebraic: in the Picard group of  Lubin-Tate space, two important
invertible sheaves become isomorphic modulo $p$. 
\end{abstract}

\maketitle


Fix a prime $p$ and and an integer $n \geq 0$, and let $K(n)$ denote the $n$th Morava $K$-theory at the prime
$p$. If $n \geq 1$, the $K(n)$-local stable homotopy category has two dualities. First, there is $K(n)$-local
Spanier-Whitehead duality $D_n(-)$. This behaves very much like Spanier-Whitehead duality in the ordinary stable category: it
has good formal properties, but it can be very hard to compute. Second, there is Brown-Comenetz duality $I_n(-)$, which
behaves much like a Serre-Grothendieck duality and, in many ways, is much more computable. One of the key features
of the $K(n)$-local category is that under some circumstances the two dualities are closely related. 

Recall that a finite spectrum $X$ is of type $n$ if $K(m)_\ast X = 0$ for $m < n$. By 
\cite{HS99}, any type $n$ spectrum has a $v_n^{p^k}$-self map; that is, there is an integer $k$ and map 
$$
\Sigma^{2p^k(p^n-1)}X \to X
$$
which induces multiplication by $v_n^{p^k}$ in $K(n)_\ast$.  In their papers on the period map and the dualizing sheaf for
Lubin-Tate space, Gross and the second author \cite{HG} wrote down the following result. Suppose $X$ is a 
type $n$-spectrum with a $v_n^{p^k}$-self map and suppose further that $p$ times the identity map of $X$ is zero. Then
if $2p > \mathrm{max}\{n^2+1,2n+2\}$ there is an equivalence in the $K(n)$-local category\footnote{The bound on $p$ is very slightly different than in \cite{HG}; see Proposition \ref{why-the-hypothesis}.}
\begin{equation}\label{GHno1}
I_nX \simeq \Sigma^{2p^{nk}r(n)+n^2-n} D_nX
\end{equation}
where $r(n) = (p^n-1)/(p-1) = p^{n-1}+\cdots+p+1$. This equivalence gives a conceptual explanation for
many of the self-dual patterns apparent in the amazing computations of Shimomura and his coauthors.
See, for example, \cite{SY}, \cite{ConferenceNotesSadofsky}, \cite{BehrensRevisited}, and \cite{Lader}.

The point of this note is to write down a linear narrative with this result at the center. In some sense, there
is nothing new here, as the key ideas can be found scattered through the literature, and other authors have
obliquely touched on this topic. A rich early example is in \S 5 of the paper \cite{DHaction} by Devinatz and the second
author, and the important paper of Dwyer, Greenlees, and Iyengar \cite{DGI} embeds many of the ideas here into a
far-reaching and beautiful theory. In another sense, however, there is
quite a bit to say, as there are any number of key technical ideas we need to access, some of which have
not quite made it into print and others buried in ways that make them hard to uncover. In any case, the result is of enough
importance that it deserves specific memorialization. 

Here is a little more detail. We fix $p$ and $n$ and let $E = E_n$ be Morava
$E$-theory for $n$ and $p$. This represents a complex oriented cohomology theory with formal group law
a universal deformation of the Honda formal group law $H_n$ of height $n$. See \S 1 for more details. As always we write
$$
E_\ast X = \pi_\ast L_{K(n)}(E \wedge X).
$$
The $E_\ast$-module $E_\ast X$ is a graded Morava module: it has a continuous and twisted 
action of the Morava stabilizer group $\GG_n = \Aut(H_n,\FF_{p^n})$. See Remark \ref{Morava-mods}.

There are two key steps to the equivalence (\ref{GHno1}). We have a $K(n)$-local equivalence
$I_nX \simeq I_n \wedge D_nX$ where $I_n = I_n(S^0)$; thus, the first step is the identification of the 
homotopy type of $I_n $, at least for $p$ large with respect to $n$. This is also due to Gross and the second author, with 
details laid out in \cite{StrickGrossHop}. The key fact is that $I_n$ is dualizable in the $K(n)$-local category; by \cite{HMS}
this is equivalent to the statement that $E_\ast I_n$ is an invertible graded Morava module and, indeed, the main result of
\cite{StrickGrossHop} (interpreting \cite{HG})  is that there is an isomorphism of Morava modules
$$
E_\ast I_n \cong E_\ast S^{n^2-n}[\det]
$$
where $S^0[\det]=S[\det]$ is a determinant twisted sphere in the $K(n)$-local category; see Remark \ref{Sdet-defined}. 
The number $r(n)$ in (\ref{GHno1}) is an artifact of the determinant; see (\ref{whyrn}).

The second key step is an analysis of the $K(n)$-local Picard group $\Pic_{K(n)}$ of equivalence classes of
invertible objects in the $K(n)$-local
category. As mentioned, we know that a $K(n)$-local spectrum $X$ is invertible if and only if $E_\ast X$
is an invertible graded Morava module. We also know that the group of invertible graded 
Morava modules concentrated in even degrees is isomorphic to the continuous cohomology group
$H^1(\GG_n,E_0^\times)$,  where $E_0^\times$ is the group of units in the ring $E_0$. Hence, if we write
$\Pic^0_{K(n)} \subseteq \Pic_{K(n)}$ for the subgroup  of objects $X$ with $E_\ast X$ in even degrees, we get
a map
$$
e:\Pic^0_{K(n)} \longr H^1(\GG_n,E_0^\times).
$$
The map is an injection under the hypothesis
$2p > \mathrm{max}\{n^2+1,2n+2\}$. See Proposition \ref{why-the-hypothesis}. This is the origin for the hypothesis
on $p$ and $n$ in the equivalence of (\ref{GHno1}): it reduces that equivalence to an algebraic calculation.

It is an observation of \cite{HMS} that the map $\ZZ \to \Pic^0_{K(n)}$ sending $k$ to
$S^{2k}$ extends to an inclusion of the completion of the integers
$$
\RC \defeq \lim_k \ZZ/(p^k(p^n-1)) \to \Pic^0_{K(n)};
$$
that is, for any $a \in \RC$ we have a sphere $S^{2a}$. (The phrase ``$p$-adic sphere''
is common here, but misleading: $\RC$ is not the $p$-adic integers. See Remark \ref{get-compl-straight}.)  Now
let $\lambda = \lim_k\ p^{nk}r(n) \in \RC$. The key algebraic result can now be deduced from
Proposition \ref{detred-alg} below: under the composition
$$
\Pic^0_{K(n)} \mathop{\longr}^{e} H^1(\GG_n,E_0^\times) \to H^1(\GG_n,(E_0/p)^\times)
$$
the spectra $S[\det]$ and $S^{2\lambda}$ map to the same element. The equivalence (\ref{GHno1})
follows once we observe that if $X$ is type $n$ and has a $v_n^{p^k}$-self map, then there is $K(n)$-local
equivalence
$$
S^{2\lambda} \wedge X \simeq \Sigma^{2p^{nk}r(n)} X.
$$
See Theorems \ref{detred} and \ref{detred-top}.

It is worth emphasizing that the algebraic result Proposition \ref{detred-alg} only requires $p > 2$; it is
the topological applications which require the more stringent restrictions on the prime. In fact, the equivalence 
of dualities in (\ref{GHno1}) can be false if the prime is small. See Remark \ref{counter-exam}. 

The plan of this note is as follows: in the first section we give some homotopy theoretic and algebraic background,
in the second section we give a discussion of the Picard group, lingering long enough to give details of
the structure of $\pico$ as a profinite $\RC$-module. See Proposition \ref{pic0-zn}. In Section 3 we discuss
the determinant and prove the key Proposition \ref{detred-alg}. In Section 4 we give some discussion 
of how Spanier-Whitehead and Brown-Comenetz duality behave in the Adams-Novikov Spectral Sequence. In the final
section, we give the homotopy theoretic applications. 

\subsection*{Acknowledgements} This project began as an attempt
to find a conceptual computation of the self-dual patterns apparent in the Shimomura-Yabe calculation \cite{SY} of
$\pi_\ast L_{K(2)}V(0)$ at $p  > 3$. Then in a conversation  with Tobias Barthel it emerged  that
there was no straightforward argument in print to prove the equivalence of dualities in (\ref{GHno1}). Later, as Guchuan
Li was working on his thesis \cite{LiThesis} it became apparent that he needed these results and, more,
there were constructions once present in the general culture that were no longer easily accessible. Thus a sequence of notes,
begun at MSRI in Spring of 2014, have evolved into this paper. Many thanks for Agn\`es Beaudry and Vesna Stojanoska for 
reading through an early draft. Others have surely written down proofs for themselves as well; for example,
Hans-Werner Henn once remarked that ``the determinant essentially disappears mod $p$,'' which is a very succinct
summary of Proposition \ref{detred-alg}.

Finally, the authors would like to the thank the referee for a very careful reading and helpful comments. 

\numberwithin{equation}{section}

\section{Some background}

In this section we gather together the basic material used in later sections. All of this is throughly
covered in the literature and collected here only for narrative continuity.

\subsection{The $K(n)$-local category} For an in-depth study of the technicalities
in the $K(n)$-local category, see Hovey and Strickland \cite{HvStr}. Other
introductions can be found in almost any paper on chromatic homotopy theory. We were especially
thorough in \cite{BGH} \S 2. 

Fix a prime $p$ and an integer $n > 0$. In order to be definite we define the $n$th Morava
$K$-theory $K(n)$ to be the $2$-periodic complex oriented cohomology theory with
coefficients $K(n)_\ast = \FF_{p^n}[u^{\pm1}]$ with $u$ in degree $-2$. The 
associated formal group law over $K(n)_0 = \FF_{p^n}$  is the unique $p$-typical formal
group law $H_n$ with $p$-series $[p]_{H_n}(x)  = x^{p^n}$. 
This is, of course, the $n$th  Honda formal group law. For $H_n$ we have
$$
v_n = u^{1-p^n} \in K(n)_{2(p^n-1)}.
$$
The $K(n)$-local category is the category of $K(n)$-local spectra. 

We also have $K(0) = H\QQ$, the rational Eilenberg-MacLane spectrum, and $K(0)$-local spectra
are the subject of rational stable homotopy theory.

We define $\GG_n = \Aut(H_n,\FF_{p^n})$ to be the group of automorphisms of the pair $(H_n,\FF_{p^n})$.
Since $H_n$ is defined over $\FF_p$, there is a splitting
$$
\Aut(H_n,\FF_{p^n}) \cong \Aut(H_n/\FF_{p^n}) \rtimes \Gal(\FF_{p^n}/\FF_{p})
$$
where the normal subgroup is the isomorphisms of $H_n$  as a formal group law
over $\FF_{p^n}$. We write $\SS_n = \Aut(H_n/\FF_{p^n})$ for this subgroup.

To get a Landweber exact homology theory which captures more than Morava $K$-theory, we use 
the Morava (or Lubin-Tate) theory $E=E_n$. This theory has coefficients
$$
E_\ast = \WW[[u_1,\ldots,u_{n-1}]][u^{\pm 1}]
$$
where again $u$ is in degree $-2$ but the power series ring is in degree $0$. The ring 
$\WW=W(\FF_{p^n})$ is the Witt vectors of $\FF_{p^n}$. 

Note that $E_0$ is a complete local ring with maximal ideal $\mm$ generated by the 
regular sequence $\{ p,u_1,\ldots,u_{n-1}\}$. We choose the
formal group law $G_n$ over $E_0$ to be the unique $p$-typical formal group law with
$p$-series 
\begin{equation}\label{gndef}
[p]_{G_n}(x) = px +_{G_n} u_1 x^p +_{G_n} \cdots +_{G_n} u_{n-1}x^{p^{n-1}}
+_{G_n} x^{p^n}.
\end{equation}
Thus $v_i = u_iu^{1-p^i}$, $1 \leq i \leq n-1$, $v_n = u^{1-p^n}$ and $v_i = 0$ if
$i > n$. Note that $G_n$ reduces to $H_n$ modulo $\mm$.

We define $E_\ast X = (E_n)_\ast X$ by
$$
E_\ast X = \pi_\ast L_{K(n)}(E \wedge X).
$$
While not quite a homology theory, as it does not take wedges to sums, it is by far our
most sensitive algebraic invariant in $K(n)$-local homotopy theory.
The group $\GG_n$ acts continuously on $E_\ast X$ making $E_\ast X$ into a
{\it Morava module}. We will be more precise on this notion below in Remark \ref{Morava-mods}. 

A basic computation gives
$$
E_0E = \pi_0L_{K(n)}(E \wedge E) \cong \map^c(\GG_n,E_0)
$$
where $\map^c$ denotes the continuous maps. See Lemma 10 of \cite{StrickGrossHop} for a proof.
The $K(n)$-local $E_n$-based Adams-Novikov Spectral Sequence now reads
\begin{equation}\label{eq:ANSS}
H^s(\GG_n,E_tX) \Longrightarrow \pi_{t-s}L_{K(n)}X.
\end{equation}
Cohomology here is continuous cohomology. 

\begin{rem}[{\bf Lubin-Tate theory}]\label{LTttheory}
The pair $(G_n,E_0)$ has an important universal property which is useful for understanding
the action of $\GG_n$.

Consider a complete local ring $(S,\mm_S)$  with $S/\mm_S$ 
of characteristic $p$. Define the groupoid  of deformations $\Def_{H_n}(S)$ to be
the category with objects $(i,G)$ where $i:\FF_{p^n} \to S/\mm_S$ is a morphism of fields and 
$G$ is a formal group law over $S$ with $q_\ast G = i_\ast H_n$. Here $q:S \to S/\mm_S$ is the quotient map.
There are no morphisms $\psi: (i,G) \to (j,H)$ if $i \ne j$ and a morphism $(i,G) \to (i,H)$ is an 
isomorphism of formal groups laws $\psi: G \to H$ so that $q_\ast \psi$
is the identity. These are the  $\star$-isomorphisms. By a theorem of Lubin and Tate \cite{LT} we know
that if two deformations are $\star$-isomorphic, then there is a unique $\star$-isomorphism 
between them. Put another way, the  groupoid $\Def_{H_n}(S)$ is discrete. Furthermore, 
$E_0$ represents the functor of $\star$-isomorphism classes of deformations:
$$
\Hom^c_{\WW}(E_0,S) \cong \pi_0 \Def_{H_n}(S).
$$
Here $\Hom^c_{\WW}$ is the set of continuous $\WW$-algebra maps.
As a universal deformation we can and do choose the formal group law $G_n$ over $E_0$
to be the  $p$-typical formal group law defined above in (\ref{gndef}).
\end{rem}

\begin{rem}[{\bf The action of the Morava stabilizer group}]\label{action-rem}
We use Lubin-Tate theory to get an action of $\GG_n$ on $E_0$. This exposition follows \cite{HKM} \S 3. 

Let $g = g(x) \in \FF_{p^n}[[x]]$ be an element in $\SS_n$. Choose any lift of $g(x)$ to $h(x)
\in E_0[[x]]$  and let $G_h$ be the unique formal group law over $E_0$ so that 
$$
h:G_h \to G_n
$$
is an isomorphism. Since $g:H_n \to H_n$ is an isomorphism over $\FF_{p^n}$, the pair $(\mathrm{id},G_h)$
is a deformation of $H_n$. Hence there is a unique $\WW$-algebra map $\phi=\phi_g:E_0 \to E_0$ and
a unique $\star$-isomorphism  $f:\phi_\ast G_n \to G_h$. Let $\psi_g$ be the composition
\begin{equation}\label{hgee}
\xymatrix{
\phi_\ast G_n \rto^f \ar@/_1pc/[rr]_{\psi_g} & G_h \rto^-h & G_n\ .
}
\end{equation}
Note that while $G_h$ depends on choices, the map $\phi_g$ and the isomorphism $\psi_g$
do not. The map $\SS_n \to \Aut(E_0)$ sending $g$ to $\phi_g$ defines the action of
$\SS_n$ on $E_0$. The Galois action
on $\WW \subseteq E_0$ extends this to an action of all of $\GG_n$ on $E_0$.
The action can be extended to all of $E_\ast$ be noting that $E_2 \cong \tilde{E}^0S^2\cong \tilde{E}^0\CP^1$
is isomorphic to the module of invariant differentials on the universal deformation $G_n$. 
See (\ref{invaru0}) below for an explicit formula.
\end{rem}

\begin{rem}[{\bf Formulas for the action}]\label{action-forms}
We make the action of $\SS_n$ a bit more precise.
By (\ref{hgee}) we have an isomorphism $\psi_g:\phi_\ast G_n \to G_n$ of $p$-typical formal 
group  laws over $E_0$. This can be written
$$
\psi_g(x) = t_0(g) +_{G_n} t_1(g)x^p +_{G_n} t_2(g)x^{p^2}  +_{G_n} \cdots. 
$$
This formula defines continuous functions $t_i:\SS_n \to E_0$.  As in Section 4.1 of \cite{HKM} we have
\begin{equation}\label{invaru0}
g_\ast u = t_0(g)u.
\end{equation}
The function  $t_0$ is a crossed homomorphism $t_0:\SS_n \to E_0^\times$; that is,
$$
t_0(gh) = [gt_0(h)]t_0(g).
$$
Since the Honda formal group is defined over $\FF_p$ we can choose the class $u$ to be invariant under the action
of the Galois group; hence $t_0$ extends to crossed homomorphism $t_0:\GG_n \to E_0^\times$ sending
$(g,\phi) \in \SS_n \rtimes \Gal(\FF_{p^n}/\FF_p) \cong \GG_n$ to $t_0(g)$. 
\end{rem}

\begin{rem}\label{facts-on-ANSS} We record here some basic useful facts about the $K(n)$-local Adams-Novikov
Spectral Sequence (\ref{eq:ANSS})  which we will use later.

The first two statements are standard and are proved using the action of the center of
$Z(\GG_n) \subseteq \GG_n$ on $E_\ast = 
E_\ast S^0$. There is an isomorphism $\ZZ_p^\times \cong Z(\GG_n)$ sendings $a \in \ZZ_p^\times$ to the
$a$-series $[a]_{H_n}(x)$ of the Honda formal group. The action of $Z(\GG_n)$ on $E_0$ is trivial and the action
on $E_\ast$ is then determined by the fact that $t_0(a) = a$; that is, $a$ acts on $u \in E_{-2}$ by multiplication
by $a$. 

1.) {\bf Sparseness:} If $t \not\equiv 0$ modulo $2(p-1)$, then $H^\ast (\GG_n,E_t) = 0$.  If $p=2$ this is not new
information. If $p > 2$ let $C \subseteq Z(\GG_n)$ be the cyclic subgroup of Teichm\"uller lifts of $\FF_p^\times$.
Then $E_t^{C} = 0$ and hence
\[
H^\ast(\GG_n,E_t) \cong H^\ast(\GG_n/C,E_t^C) = 0.
\]

2.) {\bf Bounded torsion:} Suppose $p > 2$ and suppose $2t = 2p^km(p-1)\ne 0$ with $m$ not divisible by $p$. Then  we have
\[
p^{k+1}H^\ast(\GG_n,E_{2t}) = 0.
\]
If $p=2$ write $2t = 2^k(2m+1)$. Then we have
\[
2H^\ast(\GG_n,E_{2t}) = 0\qquad \mathrm{if}\ k=1,
\]
and
\[
2^{k+1}H^\ast(\GG_n,E_{2t}) = 0\qquad \mathrm{if}\ k>1.
\]

To get these bounds,
first suppose  $p > 2$. Let $K = 1+p\ZZ_p \subseteq Z(\GG_n)$ be the torsion-free subgroup and let $x \in K$ be
a topological generator; for example, $x=1+p$. The choice of $x$ defines an isomorphism $\ZZ_p \cong K$.
Thus, there is an exact sequence
\[
\xymatrix{
0 \to H^0(K,E_{2t}) \rto & E_{2t} \rto^-{x^k-1} &E_{2t} \rto & H^1(K,E_{2t}) \to 0.
}
\]
Thus we see that $p^{k+1}H^1(K,E_{2t}) = 0$ and $H^q(K,E_{2t}) = 0$ if $q\ne 1$.
Now use the Lyndon-Hochschild-Serre Spectral Sequence 
\[
H^p(\GG_n/K,H^q(K,E_{2t})) \Longrightarrow H^{p+q}(\GG_n,E_{2t})
\]
to deduce the claim. At the prime $2$ let $x \in \ZZ_2^\times$ be an element of infinite order
which reduces to $-1$ modulo $4$ -- for example, $x=3$ -- and let $K$ be the subgroup generated by $x$. The proof
then proceeds in the same fashion.

Note that the arguments for parts (1) and (2) apply not only to $\GG_n$, but also for any closed subgroup
$G \subseteq \GG_n$ which contains the center. In fact, for part (1) we need only have $C = \FF_p^\times \subseteq G$.

3.) {\bf There is a uniform and horizontal vanishing line at $E_\infty$:} there is an integer $N$,
depending only on $n$ and $p$, so that in the Adams-Novikov Spectral Sequence (\ref{eq:ANSS}) for any spectrum
$X$
\[
E_\infty^{s,\ast} = 0, \qquad s > N.
\]
This can be found in the literature in several guises; for example, it can be put together from the material in Section
5 of \cite{DH}, especially Lemma 5.11. See \cite{BGH} \S 2.3 for references and explanation. See also \cite{BBGS} for
even further explanation. If $p-1 > n$, there is often a horizontal vanishing line at $E_2$. See Proposition \ref{e2-van-line} below.
\end{rem}

\subsection{Some local homological algebra.} Because $E_0$ is a complete local ring 
with maximal ideal  $\mm$ generated by a regular sequence, we have a variety of
tools from homological algebra. The classic paper here is Greenlees and May \cite{GreenMay}, but see also \cite{HvStr},
Appendix A for direct connections to $E_\ast(-)$. Tensor product below will mean the $\mm$-completed tensor
product.  This is one place where the notation $E_0$ gets out of hand; thus we write $R=E_0$ in this subsection.

Let $u_0 = p$ and define a cochain complex $\Gamma_\mm$ by
$$
\Gamma_\mm = (R \to R[\frac{1}{u_0}])\otimes_R (R \to R[\frac{1}{u_1}])\otimes_R \cdots
\otimes_R
(R \to R[\frac{1}{u_{n-1}}])
$$
and more generally we set
$$
\Gamma_m(M) = M\otimes_R \Gamma_\mm.
$$
Then $H^0_\mm(M) \defeq H^0\Gamma_m(M)$ is the sub-module of $\mm$-torsion and we
see that
$$
H^s\Gamma_\mm(M)  \defeq H^s_\mm(M)
$$
is the $s$th right derived functor of the $\mm$-torsion functor and thus independent of the choices. 
These are  the local cohomology groups. If $M$ is $\mm$-torsion, there is
a composite functor spectral sequence
\begin{equation}\label{UCSS-simp-1}
\Ext^p_R(M,H_\mm^q(N)) \Longrightarrow \Ext^{p+q}_R(M,N).
\end{equation}
In the case $N=R$, this spectral sequence simplifies considerably. Note that  $H^s_\mm(R) = 0$ unless $s=n$ and
\begin{equation}\label{Rminfty}
H^n_\mm(R) \defeq R/\mm^\infty \defeq R/(p^\infty,u_1^\infty,\ldots,u_{n-1}^\infty)\ .
\end{equation}
The $R$-module $R/\mm^\infty$ is an injective $R$-module and, in fact the injective hull of $R/\mm$. 
This is a consequence of Matlis duality for $(R,\mm)$; see \S 12.1 of  \cite{BrodmannSharp},
especially Definition 12.1.2 and Remark  12.1.3.

Combining this observation with the spectral sequence (\ref{UCSS-simp-1}) we have
\begin{equation}\label{UCSS-simp}
\Ext_R^{p+n}(M,R) \cong \Ext^{p}_R(M,R/\mm^\infty) \cong \begin{cases}
\ \Hom_R(M,R/\mm^\infty), &p=0;\\ \ 0,&p  \ne 0.\end{cases}.
\end{equation}

The module $R/\mm^\infty$ also arises in the theory of derived functors of completion. The
completion functor
$$
M \longmapsto \lim_k\ \big[ M \otimes_R R/\mm^k\big]
$$
is neither left nor right exact; however, it still has left derived functors $L^\mm_s(M)$. These vanish
if $s > n$ and there is an isomorphism
$$
L_n^\mm(M) \cong \lim\ \Tor^R_n(M,R/\mm^k) \cong \lim\ \Hom_R(R/\mm^k,M) \cong
\Hom_R(R/\mm^\infty,M).
$$
From this it follows that 
$$
L^\mm_s(M) \cong \Ext_R^{n-s}(R/\mm^\infty,M).
$$

\begin{rem}[\textbf{Morava modules}]\label{Morava-mods} If $X$ is a spectrum we defined 
\[
E_\ast X= \pi_\ast L_{K(n)}(E \wedge X).
\]
By \cite{HvStr}, Proposition 8.4, the $E_\ast$-module $E_\ast X$ is $L^\mm$-complete; that is, the map
$$
E_\ast X \longr L_0^\mm(E_\ast X)
$$
is an isomorphism. In particular, $E_\ast X$ is equipped with the $\mathfrak{m}$-adic topology. 

The action of $\GG_n$ on $E$ determines a continuous action of $\GG_n$ on $E_t X$. 
This  action is twisted in the sense that
if $g \in \GG_n$, $a \in E_0$ and $x \in E_t X$, then $g_\ast(ax) = g_\ast(a)g_\ast(x)$.
We will call an $L^\mm$-complete $E_0$-module with a continuous and twisted $\GG_n$ action a {\it Morava module}.
Many (if not all) of our Morava modules will actually be $\mm$-complete; that is, the natural maps
\[
M \longr L_0^\mm M \longr \lim M/\mm^k M
\]
are all isomorphisms. For example, if $M$ is $\mm$-complete, so is the induced module of continuous
map $\map^c(\GG_n,M)$. Hence the continuous cohomology of $\mm$-complete Morava modules can be 
constructed entirely in the category of $\mm$-complete Morava modules. 

The graded Morava module module $E_\ast X$ is determined by $E_0X$, $E_{1}X$, and the isomorphism of $\GG_n$-modules,
for $n \in \ZZ$,
$$
E_{t+2n}X \cong E_2^{\otimes n} \otimes_{E_0} E_tX.
$$
The $\GG_n$-action is the diagonal action.  If $n \geq 0$, $E_2^{\otimes n} = E_2 \otimes_{E_0} \cdots \otimes_{E_0} E_2$ is a
free of rank $1$ over $E_0$. If $n<0$, then $E_2^{\otimes n}$ is the dual $\GG_n$-module to $E_2^{\otimes -n}$.
This discussion gives an evident category of graded Morava modules.

We say a graded Morava module $M_\ast$ is finitely generated if it is finitely generated as a graded $E_\ast$-module or, 
equivalently, if $M_0$ and $M_1$ are finitely generated as $E_0$-modules. We also say a graded Morava module $M_\ast$
is finite if $M_0 $ and $M_{1}$ are finite. If $X$ is a finite CW spectrum then $E_\ast X$ is finitely generated. More
generally, $X$ is dualizable in the $K(n)$-local category if and only if $E_\ast X$ is finitely generated. See
Theorem 8.6 of \cite{HvStr}. If $X$ is also
of type $n$, then $E_\ast X$ is finite.
\end{rem}

Here is a key fact about Morava modules which we use often. The argument owes quite a good deal to the proof of
Lemma 5 of \cite{StrickGrossHop}. 

\begin{prop}\label{e2-van-line} Let $p-1 > n$ and let $M$ be an $\mm$-complete Morava module. Then for all $s > n^2$
\[
H^s(\GG_n,M) = 0.
\]
\end{prop}

\begin{proof}Let $S_n \subseteq \SS_n$ be the subgroup of automorphisms $g = g(x)$ of the Honda formal group
$H_n$ so that $g'(0)=1$. Then $S_n$ is a compact $p$-adic analytic group of dimension $n^2$ by  \S 3.1.2 of 
\cite{HennDuke}.  Under the assumption $p-1 > n$ the group $S_n$ is torsion-free; see Theorem 3.2.1 of \cite{HennDuke}.
By a theorem of Lazard (combine Theorems 4.4.1 and 5.1.9. of \cite{SymondsWeigel}) we may conclude
$S_n$ is a Poincar\'e duality group of dimension $n^2$ and of cohomological dimension $n^2$. Since the index
of $S_n$ in $\SS_n$ is finite and prime to $p$, the cohomological dimension of $\SS_n$ is also $n^2$. So far we have
not used the hypothesis on $M$.

Let $\Gal = \Gal(\FF_{p^n}/\FF_p)$. If $M$ is an $\mm$-complete Morava module then $M^{\SS_n}$ is a $p$-complete
twisted $\WW$-$\Gal$-module; that is, if $g \in \Gal$ and $x \in M^{\SS_n}$, then $g(ax) = g(a)g(x)$.
Now we use a version of Galois descent -- see Lemma \ref{Gal-inb-bis} below -- to conclude
\[
H^\ast(\GG_n,M) \cong H^\ast(\SS_n,M)^{\Gal}
\]
and we have the vanishing we need.
\end{proof}

\begin{lem}\label{Gal-inb-bis} Let $\Gal = \Gal(\FF_{p^n}/\FF_p)$.
Let $M$ be a $p$-complete twisted $\WW$-$\Gal$-module. Then the inclusion
$M^{\Gal} \to M$ of the invariants extends to an isomorphism of twisted $\WW$-$\Gal$-modules
\[
\WW \otimes_{\ZZ_p} M^{\Gal} \cong M .
\]
The functor $M \mapsto M^{\Gal}$ from $p$-complete twisted $\WW$-$\Gal$-modules to $p$-complete modules is exact. 
\end{lem}

\begin{proof} This can be proved using standard descent theory, but here is a completely explicit argument.

We are using the completed tensor product
\[
\WW \otimes_{\ZZ_p} N = \lim(\WW \otimes_{\ZZ_p} N)/p^k \cong \lim (\WW/p^k \otimes_{\ZZ/p^k} N/p^kN).
\]
First, since inverse limits commute with invariants, we have
\[
M^{\Gal} \cong (\lim M/p^kM)^\Gal \cong \lim (M/p^kM)^\Gal .
\]
Next, the map $M^{\Gal} \to (M/p^kM)^{\Gal}$ factors as
\[
M^{\Gal} \longr M^{\Gal}/p^kM^{\Gal} \longr (M/p^kM)^{\Gal}
\]
with the second map an inclusion. This yields isomorphisms
\[
M^\Gal \cong \lim (M^{\Gal}/p^kM^{\Gal}) \cong \lim(M/p^kM)^{\Gal}.
\]
Since $\WW$ is a finitely generated free $\ZZ_p$ module
\[
\WW \otimes_{\ZZ_p} M^{\Gal} \cong \lim (\WW/p^k \otimes_{\ZZ_{p^k}} M^{\Gal}/p^kM^{\Gal})
 \cong \lim(\WW/p^k \otimes_{\ZZ_{p^k}}(M/p^kM)^{\Gal}).
\]
Finally, $\ZZ/p^k \to \WW/p^k$ is Galois with Galois group $\Gal$ we have
\[
\WW/p^k \otimes_{\ZZ/p^k} (M/p^kM)^{\Gal} \cong M/p^kM.
\]
The exactness statement follows from the fact that $\WW$ is a free and finitely
generated $\ZZ_p$-module, so $N \mapsto \WW \otimes_{\ZZ_p} N$ is exact. 
\end{proof}

\section{Picard groups}

The point of this section is to develop enough technology to pave the way for the key Proposition \ref{detred-alg}.

\subsection{Some basics} Let $\Pic_{K(n)}$ denote the $K(n)$-local Picard group 
of weak equivalence classes of invertible elements. Here is an observation
from \cite{HMS}.  If $X \in \Pic_{K(n)}$, then $K(n)_\ast X$
is an invertible $K(n)_\ast$-module and, since $K(n)_\ast$ is a graded field, it follows
that $K(n)_\ast X$ is of rank $1$ over $K(n)_\ast$. From this it follows from Proposition
8.4 of \cite{HvStr}  that $E_\ast X$ is also free of rank $1$ over $E_\ast$. 

\begin{rem}\label{alg-pic-zero} Let $\Pic^0_{K(n)} \subseteq \Pic_{K(n)}$ be the subgroup of
index  $2$ generated by the elements $X$ with $E_\ast X$ in even degrees. Then $E_0X$ is 
free of rank $1$ over $E_0$. If we choose a generator $a \in E_0X$ then we can 
define a crossed homomorphism $\phi:\GG_n \to E_0^\times$ by the formula
\[
ga = \phi(g)a, \qquad g \in \GG_n.
\]
This defines a homomorphism
\begin{equation}\label{exotic-e}
e:\Pic^0_{K(n)} \longr H^1(\GG_n,E_0^\times)
\end{equation}
to the algebraic Picard group of invertible Morava modules. We write $\kappa_n$ for the kernel of $e$; this is the
subgroup of {\it exotic elements} in the Picard group.

{\bf Notation:} Both $\Pic_{K(n)}$ and $H^1(\GG_n,E_0^\times)$ are abelian groups where the group operation
is written as multiplicatively; thus $e(X \wedge Y) = e(X)e(Y)$.
\end{rem}

The following result explains the hypothesis on the prime in the equivalence of dualities (\ref{GHno1}). This
appears in the literature in various guises; the exact criterion on the prime depends on the setting, but the 
proof is always  the same as in Theorem 5.4 of \cite{HSad}.

\begin{prop}\label{why-the-hypothesis} Suppose $2p > \mathrm{max}\{n^2+1,2n+2\}$. Then $\kappa_n = 0$ and the
map $e$ is an injection.
\end{prop}

\begin{proof} Suppose $X$ and $Y$ are two invertible spectra so $E_0(X) \cong E_0(Y)$
as Morava modules.  Let $D_nY = F(Y,L_{K(n)}S^0)$ be the $K(n)$-local Spanier-Whitehead dual of $Y$. 
Then $D_nY$ is the inverse of $Y$ in $\pico$; hence, $E_0 D_nY$ is the inverse of $E_0Y$ as an invertible Morava
module. It follows that $E_\ast (X \wedge D_nY) \cong E_\ast S^0$ as Morava modules and we need only show that the
class
\[
\iota \in H^0(\GG_n, E_0 (X \wedge D_nY))
\]
determined by this isomorphism is a permanent cycle. The differentials will lie in subquotients of
\[
H^{s+1}(\GG_n,E_s), \quad s \geq 1.
\]
Under the hypotheses here, we can now
apply the sparseness result of part (1) of Remark \ref{facts-on-ANSS} and the horizontal vanishing line of
Proposition \ref{e2-van-line}. The second of these requires $p-1 > n$ and the first requires $2(p-1)+1 > n^2$.
Combined they imply that all differentials on $\iota$ land in zero groups.
\end{proof}

\begin{rem}\label{Toda}
A more sophisticated variation of the argument used to prove Proposition \ref{why-the-hypothesis} will also
show that $e$ is surjective under the same hypotheses on $p$ and $n$. See \cite{Piotr} for details. Here is an outline.

Let $M$ be an invertible graded Morava module and let $M^\vee = \Hom_{E_0}(M,E_0)$ with conjugation $\GG_n$-action; see 
Remark \ref{E-UCSS-note} on why this action arises.The essential
idea is to use a Toda-style obstruction theory with successively defined obstructions
\[
\theta_s \in H^{s+2}(\GG_n,E_{s} \otimes_{E_0} M \otimes_{E_0} M^\vee) \cong H^{s+2}(\GG_n,E_{s}),\quad s \geq 1,
\]
to finding such an $X$ with $E_0 X \cong M$. Such an obstruction theory can be constructed
using Toda's techniques \cite{toda-obst} or a linearized version of the vastly more complex
obstruction theory of \cite{GoerssHopkins}. These obstruction groups 
 will vanish if $2p > \mathrm{max}\{n^2+1,2n+2\}$. 
\end{rem}

The basic example of an element in $\Pic^0_{K(n)}$ is the localized $2$-sphere $L_{K(n)}S^2$.
Since $E_0S^2 \cong E_{-2}$ we can choose $u \in E_{-2}$ as the generator and the associated crossed homomorphism is 
$t_0:\GG_n \to E_0^\times$. See (\ref{invaru0}). 

We next explore the underlying algebraic structure of the Picard group $\pico$; in particular, it is a profinite abelian group and
continuous module over the rather unusual completion of the integers
\begin{equation}\label{defn-RC}
\RC \defeq \wcom\ .
\end{equation}
The canonical isomorphisms $\ZZ/p^k(p^n-1) \cong \ZZ/p^k \times \ZZ/(p^n-1)$ assemble to give a continuous isomorphism
of rings
\[
\RC \cong \ZZ_p \times  \ZZ/(p^n-1).
\]
See Remark \ref{get-compl-straight} for more thoughts on the ring $\RC$. 

The number $p^n-1$ appears in a number of ways in $K(n)$-local homotopy theory; for example,
the element $v_n = u^{-(p^n-1)} \in E_{2(p^n-1)}$. We explore that observation
more in Remark \ref{closer-to-HS} below. In this context, however, 
the ring $\RC$ arises for a much more basic reason.

\begin{lem}\label{whyRC} Let $(S,\mm_S)$ be a complete local ring with residue field $\FF_{p^n} \cong S/\mm_s$.
The abelian group structure on the group of units $S^\times$ extends to a continuous $\RC$-module structure
in the topology given by the isomorphism $S^\times \cong \lim (S/\mm^k)^\times$. 
\end{lem}

\begin{proof} For any $a \in \RC$ and any $x \in S^\times$ we must define an element $x^a \in 
S^\times$. Furthermore if $a = n \in \ZZ$, then we need $x^a = x^n$.

Since $(S/\mm)^\times \cong \FF_{p^n}^\times$, any $x \in S^\times$ has the property
that $x^{p^n-1} \equiv 1$ modulo $\mm_S$ and, hence, that 
\[
x^{p^k(p^n-1)} \equiv 1\quad \mathrm{mod}\quad \mm^{k+1}.
\]
Let $a \in \wcom \in \RC$. For each integer $k \geq 0$ choose an integer $a_k$ so that
$a_k \equiv a \in \ZZ/p^k(p^n-1)$. Then the elements
\[
x^{a_k} \in (S/\mm^{k+1})^\times
\]
define an element $x^a \in S^\times \cong \lim (S/\mm^{k+1})^\times$ as needed.
\end{proof}

The basic application of Lemma \ref{whyRC} is to the ring $S = E_0$. This further
implies that the continuous cohomology group $H^1(\GG_n,E_0^\times)$ is a continuous module over $\RC$. 

It turns out we can show $\Pic^0_{K(n)}$ is also a profinite module over $\RC$ and the evaluation map 
\[
e:\pico \longr H^1(\GG_n,E_0^\times)
\]
is a continuous map of  $\RC$ modules. This can be deduced from Proposition 14.3.d of \cite{HvStr}, which
in turn depends heavily on \cite{HMS}. The argument given here is essentially the same, but packaged
to emphasize the role of the group $\kappa_n$ of exotic elements in the Picard group
and the cohomology group $H^1(\GG_n,E_0^\times)$. We'll give a proof in Proposition \ref{pic0-zn} below.

\begin{rem}\label{tow-gen-Moore} Using nilpotence technology derived from \cite{HS99} and working as in \cite{HvStr} \S  4
we can choose a sequence of ideals
$J(i) \subseteq \mm \subseteq E_0$ and spectra
$S/J(i)$ with the following properties:
\begin{enumerate}

\item $J(i+1) \subseteq J(i)$ and $\cap\ J(i) = 0$;

\item $E_0/J(i)$ is finite;

\item $E_0(S/J(i)) \cong E_0/J(i)$ and there are maps $q: S/J(i+1) \to S/J(i)$ realizing the quotient map
$E_0/J(i+1) \to E_0/J(i)$;

\item there are maps $\eta=\eta_i:S^0 \to S/J(i)$ inducing the quotient map $E_0 \to E_0/J(i)$ and $q\eta_{i+1} = \eta_i:
S/J(i) \to S/J(i)$; 

\item if $X$ a finite type $n$-spectrum, then the map $X \to \holim\ X \wedge S/J(i)$ induced by the maps
$\eta$ is an equivalence; and, 

\item the $S/J(i)$ are $\mu$-spectra; that is, there are maps $\mu: S/J(i) \wedge S/J(i) \to S/J(i)$
so that $\mu(\eta \wedge 1) = 1:S/J(i) \to S/J(i)$. 
\end{enumerate} 

They also prove that items (1)-(5) characterize the tower $\{S/J(i)\}$ up to equivalence
in the pro-category of towers under $S^0$. See
Proposition 4.22 of \cite{HvStr}. 

Hovey and Strickland choose the $J(i)$ with the property that there are positive integers $a_0,a_1,\ldots,a_{n-1}$
(depending on $i$) so that
\[
J(i)  = (p^{a_0},u_1^{a_1},\ldots,u_{n-1}^{a_{n-1}}).
\]
They don't quite say it explicitly, but in their construction they choose the $a_i$, $i \geq 1$, to be powers of $p$. 
\end{rem}

\begin{rem}\label{pic-top} Let $G(i) \subseteq \Pic^0_{K(n)}$ be the set of equivalence classes $X$ which
can be given a $K(n)$-local equivalence
\[
X \wedge S/J(i) \simeq S/J(i).
\]
Item (6) of Remark \ref{tow-gen-Moore} is used to show $G(i+1) \subseteq G(i)$. 
By Proposition 14.2 of \cite{HvStr} $G(i)$ is a finite index subgroup and $\cap\ G(i) = \{L_{K(n)}S^0\}$; thus the subgroups
$G(i)$ define a separated profinite topology on $\Pic^0_{K(n)}$.
\end{rem}

\begin{lem}\label{ev-is-cont} The evaluation map
\[
e:\Pic^0_{K(n)} \longr H^1(\GG_n,E_0^\times)
\]
is a continuous homomorphism of profinite abelian groups.
\end{lem}

\begin{proof} Parts (1) and (2) of Remark \ref{tow-gen-Moore}
imply the ideals $\{ J(i)\}$ define the same topology as $\{ \mm^k\}$ on $E_0$. By Part (2) of  Remark \ref{tow-gen-Moore}
we have $H^0(\GG_n,E_0/J(i)^\times)$ is finite, hence $\lim^1 H^0(\GG_n,E_0/J(i)^\times) = 0$ and the
$\lim$-$\lim^1$ short exact sequence in cohomology gives
\[
H^1(\GG_n,E_0^\times) \cong \lim H^1(\GG_n,E_0/J(i)^\times). 
\]
By definition the evaluation map $e$ factors
\[
e:\Pic^0_{K(n)}/G(i) \longr H^1(\GG_n,E_0/J(i)^\times)\ .
\]
The result follows.
\end{proof}

\begin{rem}\label{basics-prof} Let $A = \lim A_i$ be a profinite abelian group. 
Let $U_i \subseteq A$ be the open and closed subgroups with $A/U_i \cong A_i$. If $B \subseteq A$ is a subgroup,
let $V_i = B \cap U_i$ and $B_i = B/V_i$. Then $B_i \subseteq A_i$ is a subgroup and $B \to \lim B_i$ is the closure of $B$ in 
$A$; thus $B$ is closed if and only if $B = \lim B_i$. Note that $B$ is then profinite. If $B$ is closed then
$A/B = \lim A_i/B_i$ since $\lim^1 B_i = 0$; hence $A/B$ is also profinite. 

From this we see that if $f:A \to B$ is a continuous map of profinite abelian groups then both the kernel and
image of $f$ are closed. 
\end{rem}

\begin{rem}\label{when-zn} We now write down a criterion for identifying $\RC$-modules. 
Let $A = \lim A_i$ be a profinite abelian group. Since the groups $A_i$ are discrete,
$A$ is a continuous $\RC$-module if and only if the following criterion holds:
\begin{quote}
(Criterion 1)  Let $x = (x_i) \in \lim A_i = A$. Then for all $i$ there is an integer $N_i$ so that $p^{N_i}(p^n-1)x_i = 0 \in A_i$.
\end{quote}

This can be deduced directly from the definitions. 

From Criterion 1 it follows that if $B \subseteq A$ is a closed subgroup and $A$ is $\RC$-module, then both
$B$ and $A/B$ are $\RC$-modules. There is a partial converse as well. We will say
a profinite abelian group $A=\lim A_i$ is a pro-$p$-group if for all $i$ the group $A_i$ is a finite abelian $p$-torsion group. 
Note that by Criterion  1 any abelian pro-$p$-group is a $\RC$-module.
Suppose
\[
0 \to A \to B \to C \to 0.
\]
is a short exact sequence of profinite abelian groups, $C$ is an $\RC$-module and $A$ is pro-$p$-group.
Then $B$ is a $\RC$-module. 

To prove this converse statement,  first note that by reindexing if necessary we can
write the short exact sequence as the inverse limit of a tower of short exact sequences
\[
0 \to A_i \to B_i \to C_i \to 0.
\]
Let $b = (b_i) \in B$ and  $(c_i)$ its image in $C$. Then for all $i$ there is an integer $N_i$ so that
$p^{N_i}(p^n-1)c_i = 0$, whence $p^{N_i}(p^n-1)b_i \in A_i$. By hypothesis there is an integer $M_i$ so that
$p^{M_i}A_i = 0$, or $p^{(M_i+N_i)}(p^n-1)b_i = 0$ as needed. 
\end{rem}

\begin{lem}\label{extend-zn} Suppose $A \to B \to C$ is a sequence of profinite abelian groups so that
$A$ is a pro-$p$-group and $C$ is a $\RC$-module. Then $B$ is a $\RC$-module and the maps
are all $\RC$-module maps.
\end{lem}

\begin{proof} By Remark \ref{basics-prof}, we may assume that the sequence is short exact and then
we can apply the previous remark.
\end{proof}

We now get to our core observation. 

\begin{prop}\label{pic0-zn} The profinite abelian group $\pico$ is a $\RC$-module and the evaluation map
\[
e:\pico \longr H^1(\GG_n,E_0^\times)
\]
is a continuous map of $\RC$-modules. 
\end{prop}

\begin{proof} By Lemma \ref{ev-is-cont} we need only show $\pico$ is a continuous $\RC$-module. 
We apply Lemma \ref{extend-zn} to the exact sequence
\[
0 \to \kappa_n \to \pico \to H^1(\GG_n,E_0^\times). 
\]
For this to work we need to know that $\kappa_n$ is pro-p-group. Since $e$ is a continuous map between
profinite groups, its kernel $\kappa_n$ is a profinite group.  So it suffices to show that there is an integer $K$
so that $p^K\kappa_n = 0$. The argument is standard; it can be deduced from \cite{HMS} and see also, for example, 
\cite{DrewHeardThesis}. Here are the details.

If $X \in \kappa_n$, then $E_0 X \cong E_0$ as a Morava module. A choice of isomorphism defines a class
\[
\iota_X \in H^0(\GG_n,E_0X).
\]
The choice of $\iota_X$ is unique up to the group of automorphisms $\Aut(E_0)\cong \ZZ_p^\times$ of  $E_0$ as a
Morava module; therefore, we have that $X \simeq L_{K(n)}S^0$ if and only $\iota_X$ is a permanent cycle in
the Adams-Novikov Spectral Sequence for $X$. Define a filtration of $\kappa_n$ by setting 
\[
F_s = \{ X \in \kappa_n\ |\ d_r(\iota_X) = 0,\ r < s\ \}.
\]
Then $F_{s+1} \subset F_{s}$ and $d_s$ defines an injection
\[
F_s/F_{s+1} \longr E_s^{s,s-1}
\]
By part (3) of Remark \ref{facts-on-ANSS} we have a uniform
vanishing line for  the  Adams-Novikov Spectral Sequence; hence, this filtration is finite. To finish the proof we need to show
$F_s/F_{s-1}$ is of bounded $p$-torsion for all $s$.

The group $E_s^{s,s-1}$ is a subquotient of $H^s(\GG_n,E_{s-1})$.  We now use parts (1) and (2) of
Remark \ref{facts-on-ANSS} to conclude $H^s(\GG_n,E_{s-1})$ is bounded $p$-torsion and the
result follows.
\end{proof}

\begin{rem}\label{bound-K}It is possible to give a crude upper bound on the integer $K$ so that $p^K\kappa_n=0$ by 
summing the upper bounds of the integers $k_s$ so that $p^{k_s}$ annihilates $H^s(\GG_n,E_{s-1})$. 

Furthermore,  in our admittedly very few known examples, $\kappa_n$ is a finite group. For $n=1$ and $p=1$,
see \cite{HMS}. For $n=2$ and $p=3$, see \cite{GHMRPic}. For $n=2$ and $p=2$, the cohomology calculations
of \cite{BGH} show $\kappa_2$ must be finite.
\end{rem}

\begin{rem}\label{p-adic-spheres} The homomorphism $\ZZ \to \pico$ sending $n$ to $S^{2n}$ extends to a
homomorphism of profinite $\RC$-modules $\RC \to \pico$, which we write as $a \mapsto S^{2a}$. 
This first appeared in \cite{HMS}; the argument there can be distilled from the one given here;
in particular it uses the tower $\{ S/J(i) \}$.

Each of the spectra $S/J(i)$ is a type $n$ complex and, hence, there
is an integer $N_i$ so that $S/J(i)$ has $v_n^{p^{N_i}}$-self  map. Then
$S^{2p^{N_i}(p^n-1)} \in G(i)$. We can arrange for $N_{i+1} \geq N_i$ and $N_i \to \infty$ as $i \to \infty$.
Then the map $\ZZ \to \pico$ sending $n$ to $S^{2n}$ extends to map
\[
\ZZ \to \RC \cong \lim \ZZ/p^{N_i}(p^n-1) \to \lim \pico/G(i) \cong \pico.
\]
This procedure writes $S^{2a}$ as a homotopy inverse limit in the $K(n)$-local category
\[
S^{2a} \simeq \holim S^{2a_i} \wedge S/J(i)
\]
where $\{ a_i \}$ is a sequence of integers so that $a_i \equiv a$ modulo $2p^{N_i}(p^n-1)$. The
transition maps in this tower are given by the maps
\begin{equation}\label{transition-maps}
\xymatrix{
S^{2a_{i+1}} \wedge S/J(i+1) \rto^-q & S^{2a_{i+1}} \wedge S/J(i) \ar[r]^-{v} & S^{2a_i} \wedge S/J(i)
}
\end{equation}
where we have write $a_{i+1} = a_i + 2m_ip^{N_i}(p^n-1)$ and $v$ is a $v_n^{p^{N_i}m_i}$-self map. 
\end{rem}

We can now record the following result. 

\begin{prop}\label{reduced-by-vn} Let $X$ be a type $n$ complex with $v_n^{p^k}$-self map. Let $a \in \RC$
and let $a_k$ be any integer so that $a_k \equiv a$ modulo $p^k(p^n-1)$. Then there is $K(n)$-local
equivalence
\[
S^{2a} \wedge X \simeq \Sigma^{2a_k}X.
\]
\end{prop}

\begin{proof} By Proposition 14.3.c of \cite{HvStr} the subgroup
\[
G_X = \{P \in \pico\ |\ P \wedge X \simeq X\ \} \subseteq \pico
\]
is of finite index, so there is an $i$ so that $G(i) \subseteq G_X$. In particular, if $S/J(i)$ has 
$v_n^{p^{N_i}}$ self map, then $S^{2p^{N_i}(p^n-1)} \in G_X$. It follows that if $a_i$ is an integer so that 
$a \equiv a_i$ modulo $2p^{N_i}(p^n-1)$ then there is $K(n)$-local equivalence
\[
S^{2a} \wedge X \simeq S^{2a_i} \wedge X.
\]
If we assume $N_i \geq k$, we can then use the $v_n^{p^k}$-self map of $X$ to write $ S^{2a_i} \wedge X
\simeq S^{2a_k} \wedge X$.
\end{proof}

\begin{rem}\label{reduced-by-vn-alg} Proposition \ref{reduced-by-vn} has the following algebraic analog.
Let $a \in \RC$ and let $a_k$ be any integer so that $a \equiv a_k$ modulo $p^k(p^n-1)$.
Then there is an isomorphism of Morava modules
\[
E_0(S^{a})/\mm^k \cong E_0(S^{a_k})/\mm^k
\]
where $\mm \subseteq E_0$ is the maximal ideal.
This can be also be proved by a calculation with crossed homomorphisms. 

Note that
\[
E_\ast S^{2a} \cong \lim\ \Sigma^{2a_i} E_\ast/J(i)
\]
is a free $E_\ast$-module generated by an element $x_a$ in degree $0$ which reduces to the canonical
generator $u^{a_k} \in \Sigma^{2a_k} E_0/J(I)$. We conclude that for $g \in \GG_n$
$$
g_\ast x_a = t_0(g)^a x_a.
$$
The function $t_0$ was defined in Remark \ref{action-forms}.
\end{rem}

\begin{rem}\label{get-compl-straight}The ring $\RC = \wcom$ is a somewhat unusual completion of the integers; 
thus, it might be  worthwhile to analyze it and various of the elements therein.
The split short exact sequences
$$
\xymatrix{
0 \rto &\ZZ/p^k \ar[rr]^-{\times (p^n-1)} && \ZZ/p^k(p^n-1) \rto& \ZZ/(p^n-1) 
\rto &0
}
$$
assemble to give a split short exact sequence
$$
\xymatrix{
0 \rto &\ZZ_p \ar[rr]^-{\times (p^n-1)} &&\RC \rto& \ZZ/(p^n-1)
\rto &0\ .
}
$$
To give a specific element of order $p^n-1$ in $\RC$, notice that 
\begin{align*}
p^{nk} = p^{n(k-1)}(p^n-1) + p^{n(k-1)}.
\end{align*}
and, hence, $p^{nk}  \equiv p^{n(k-1)}$ modulo $p^{n(k-1)}(p^n-1)$. Thus the elements
$$
p^{nk} \in \ZZ/p^{nk}(p^n-1)
$$
assemble to give an element $\alpha$ of order $p^n-1$ in the inverse limit. We will use the notation
\[
\alpha = \lim p^{nk} \in \RC.
\]
In using this limit notation, we must be wary of the different completions. We have  $0 = \lim p^{nk} \in \ZZ_p$, but
$0 \ne \alpha = \lim p^{nk} \in \RC$.

We can also analyze an element of $\RC$ which will appear in the main algebraic
result, Proposition \ref{detred-alg}. Recall  that $r(n) = (p^n-1)/(p-1)$. Define
$$
\lambda = \lim_k\ p^{nk}r(n) \in \lim_k \ZZ/p^{nk}(p^n-1) = \RC.
$$
Then
$$
\lambda  = \alpha r(n) \in  \RC.
$$
Since $r(n)(p-1) = p^n-1$ we have that $\lambda$ is torsion of exact order $p-1$.
\end{rem}

\begin{rem}\label{closer-to-HS} In the spirit of \cite{HSad}, we could begin with the Johnson-Wilson theory $E(n)$
with
\[
E(n)_\ast \cong \ZZ_{(p)}[v_1,\ldots,v_{n-1},v_n^{\pm 1}].
\]
This is a complex oriented theory with obvious $p$-typical formal group law; compare (\ref{gndef}). 
There is a map $E(n) \to E_n$ of complex oriented theories inducing an equivalence
\[
L_{K(n)}E(n) \simeq E_n^{hF}
\]
where $E^{hF}$ is the Devinatz-Hopkins fixed points \cite{DH} with respect to
\[
F = \FF_{p^n}^\times \rtimes \Gal(\FF_{p^n}/\FF_p) \subseteq \GG_n.
\]
Since $E(n)$ and $E^{hF}$ are exactly $v_n$-periodic, we have a degree function
\[
\mathrm{deg}:\Pic_{K(n)} \longr \ZZ/2(p^n-1)
\]
sending $X$ to the degree of generator of $\pi_\ast L_{K(n)}(E(n) \wedge X) \cong E_\ast^{hF}X = (E_\ast X)^F$. 
Proposition \ref{pic0-zn} now implies the kernel of this map is an abelian pro-$p$-group and, in particular,
a continuous module over $\ZZ_p$. This is exactly the statement of  Proposition 14.3.d of \cite{HvStr}.

If $X$ has degree zero, then $(E^{hF})_0X$ is an invertible $(E^{hF})_0E^{hF}$-comodule in some appropriate category
of completed $(E^{hF})_0$-modules, but since $F$ is not normal in $\GG_n$ this category of comodules is not
obviously a category of modules over some group. Thus the analog of the evaluation map (\ref{exotic-e})
is a bit harder to define and analyze.
\end{rem}

\section{The determinant and the reduction modulo $p$}

In (\ref{invaru0}) we introduced the crossed homomorphism $t_0:\GG_n \to E_0^\times$ defining the
invertible Morava module $E_0S^2$. The determinant homomorphism
defines another invertible module, this time not arising from the unlocalized stable category. The purpose
of the section is to write down exactly how these two modules are related. See Proposition \ref{detred-alg},
Theorem \ref{large-p-det}, and Proposition \ref{pic-mod-p-n=2}.

\begin{rem}[{\bf The determinant}]\label{det-defined} Let  $\End (H_n/\FF_{p^n})$ be
the endomorphism ring of the Honda formal group law over $\FF_{p^n}$. By Theorem A2.2.18 of \cite{RavVeryGreen},
there is an isomorphism
$$
\WW\langle S\rangle/(S^n-p,Sa=a^\phi S) \cong \End (H_n/\FF_{p^n}) 
$$
where $S=x^p  \in \End (H_n/\FF_{p^n})$ and $a \mapsto a^\phi$ is the lift of the Frobenius
of $\FF_{p^n}$ to $\WW = W(\FF_{p^n})$.

The right action of 
$$
\SS_n = \Aut(H_n/\FF_{p^n}) \subseteq \End (H_n/\FF_{p^n}) 
$$
on the endomorphism ring determines a composition
\[
\xymatrix{
\SS_n \rto^-A & \mathrm{Gl}_n(\WW) \rto^-{\mathrm{det}} & \WW^\times
}
\]
where the last map is the determinant. The matrix representation of a typical element
$$
g=a_0 + a_1S + \cdots a_{n-1}S^{n-1} \in \SS_n
$$
with respect to the $\WW$-basis $\{S^i\ \ |\ 0 \leq i \leq n-1\}$ of the endomorphism ring is
\[
A(g) = \begin{bmatrix} 
    a_{0} & pa_{n-1}^\phi & \dots &pa_{1}^{\phi^{n-1}}\\
     a_{1} & a_{0}^\phi & \dots & pa_{2}^{\phi^{n-1}}\\
    \vdots &\vdots &\ddots &\vdots \\
    a_{n-1} &  a_{n-2}^\phi&\dots& a_{0}^{\phi^{n-1}} 
    \end{bmatrix}\ .
\]
From this we conclude that 
\begin{equation}\label{whyrn}
\det A(g) \equiv a_0^{r(n)}\qquad \mathrm{modulo}\ p
\end{equation}
with $r(n) = (p^n-1)/(p-1)$. In addition, we have that $\det A(g) \in \ZZ_p^\times$; hence,
we get a homomorphism
$$
\SS_n \rtimes \Gal(\FF_{p^n}/\FF_p) \longr \ZZ_p^\times \times \Gal(\FF_{p^n}/\FF_p)
\mathop{\longrightarrow}^{p_1} \ZZ_p^\times
$$
which we call $\det$. This defines an action of $\GG_n$ on $\ZZ_p$ which we write 
$\ZZ_p[\det]$ and if $M$ is any Morava module we define a new Morava module
$$
M[\det] = M \otimes \ZZ_p[\det]
$$
with the diagonal action. The invertible module $E_0[\det]$ determines the element of
$H^1(\GG_n,E_0^\times)$ defined by the homomorphism given by the composition
$$
\xymatrix{
\GG_n \rto^-{\det} &\ZZ_p^\times  \rto^-{\subseteq} & E_0^\times.
}
$$
\end{rem}

\begin{rem}[\bf Realizing the determinant]\label{Sdet-defined}
There is a canonical determinant sphere $S[\det] \in \Pic_{K(n)}$ with $E_0S[\det] \cong E_0[\det]$ as Morava modules.
Complete details are in \cite{BBGS}, but the argument is essentially due to the second author here. Assume $p > 2$. Here
is an outline. Since the classifying space $B\ZZ_p$ is a model for the $p$-completion $S^1$, the evident action
of $\ZZ_p^\times$ on $\ZZ_p$ and suspension define an action of $\ZZ_p^\times$ on
the completed sphere spectrum $S^0_p$. The determinant $\det:\GG_n \to \ZZ_p^\times$ then
gives an action of $\GG_n$ on $S_p^0$. We then can define 
$$
S[\det] = (E\wedge S_p^0)^{h\GG_n}
$$
where we use the diagonal action. The prime $2$ requires a more delicate argument. 
\end{rem}

\begin{rem}\label{Teich} Recall from Remark \ref{get-compl-straight} we have defined numbers
$\alpha = \lim p^{nk}$ and $\lambda = \lim p^{nk}r(n)$ in $\RC$. Let $(S,\mm_k)$ be a complete local ring with
$S/\mm_k \cong \FF_{p^n}$. For any $x \in S^\times$, $x^{p^n} \equiv x$ modulo $\mm_S$, so 
\[
x^{p^{nk}} \equiv x^{p^{n(k-1)}}\quad \mathrm{modulo}\quad \mm^k_S.
\]
Thus
\[
T_1(x) = \lim_k x^{p^{nk}} = x^\alpha \in S
\]
exists and $T_1(x) \equiv x$ modulo $\mm_S$. 
If $x \equiv y$ modulo $\mm_S$, then 
$x^{p^{nk}} \equiv y^{p^{nk}}$ modulo $\mm^k_S$
so we have that $T_1$ factors 
\[
\xymatrix{
S^\times \rto & \FF_{p^n}^\times \rto^{T} & S^\times
}
\]
and $T(xy) = T(x)T(y)$. In particular, $T(x)^{p^n-1} = 1$. The element $T(x)$ is the Teichm\"uller lift of $x$. In
a radical but standard abuse of notation, we often simply write $x$ for $T(x)$. 

We can apply these constructions to $H^1(\GG_n,E_0^\times)$ as well. If $g \in \SS_n$ then
$g = g(x) \in \FF_{p^n}[[x]]$ with leading coefficient $g'(0) \in \FF_{p^n}^\times$. 
Since $t_0(g) \equiv g'(0)$ modulo $\mm$ we have
\begin{equation}\label{fund-alpha-det}
t_0(g)^\alpha = T_1(t_0(g)) = g'(0)^\alpha = g'(0) \in E_0^\times
\end{equation}
and since $\det(g) = g'(0)^{r(n)}$ modulo $\mm$ we have
\[
t_0(g)^\lambda = (t_0(g)^\alpha)^{r(n)} = g'(0)^{r(n)} = \det(g)^\alpha \in E_0^\times.
\]
for all $g$, or $t_0^\lambda = \det^\alpha$ as a crossed homomorphisms on $\SS_n$. Since both functions are constant
on the Galois group, we also have $t_0^\lambda = \det^\alpha$ as crossed homomorphisms on $\GG_n$.
\end{rem}

We now turn to an examination of $X \wedge V(0)$ where $X \in \Pic_{K(n)}$. Here $V(0)$ is the
mod $p$ Moore spectrum.

\begin{lem}\label{fund-seq-exp-alg}Let $S$ be a torsion-free complete local ring with maximal ideal $\mm_S$ and suppose
$S/m_S$ is of characteristic $p>2$.Then there is a short exact sequence
\begin{equation*}
\xymatrix@C=20pt{
0 \rto &S \ar[rr]^-{\mathrm{exp}(p-)} && S^\times \rto & (S/p)^\times \rto &1.
}
\end{equation*}
\end{lem}

\begin{proof} An element $x \in S$ is a unit if and only if $0 \ne x \in S/\mm_S$. Thus the map $S^\times \to (S/p)^\times$
is onto with kernel $K$ consisting of all elements of the form $1+py$. Note that if $x \in S$, the
power series $\exp(px) = \sum_{n=0}^\infty p^nx^n/n!$ converges since $p > 2$. This defines a homomorphism
\[
\exp(p-):S\ \longr K
\]
with inverse sending $1+py$ to $(1/p)\log(1+py)$. 
\end{proof} 

If we set $S=E_0$ we get an exact sequence in cohomology
\begin{equation}\label{fund-seq-exp}
\xymatrix@C=20pt{
H^1(\GG_n,E_0) \ar[rr]^{\exp(p-)} && H^1(\GG_n,E_0^\times) \rto & H^1(\GG_n,(E_0/p)^\times)\ .
}
\end{equation}

\begin{rem}[{\bf Definition of $\zeta=\zeta_n$}]\label{define-zeta} Let $p > 2$. The Teichm\"uller homomorphism
(see Remark \ref{Teich}) $T:\FF_p^\times \to \ZZ_p^\times$ defines a splitting of the projection
$\ZZ_p^\times \to \FF_p^\times$ and determines an isomorphism
\[
1+ p\ZZ_p \times \FF_p^\times \mathop{\longr}^{\cong} \ZZ_p^\times.
\]
We'll confuse elements in $\FF_p^\times$ with their Teichm\"uller lifts. 

Let $g \in \SS_n$. Since $\det(g) \equiv g'(0)^{r(n)}$ mod $p$, we have
\[
\theta(g) \defeq g'(0)^{-r(n)}\det(g) \in 1+p\ZZ_p.
\]
We now define $\zeta = \zeta_n:\SS_n \to \ZZ_p$ to be the composite
\[
\xymatrix{
\SS_n \rto^-\theta & 1+p\ZZ_p \rto^-\ell_-{\cong} & \ZZ_p
}
\]
where $\ell(1+py)  = (1/p)\log(1+py)$. As with the determinant, $\zeta$ extends to a homomorphism $\zeta:\GG_n \to \ZZ_p$
using the projection:
\[
\xymatrix{
\SS_n \rtimes \Gal(\FF_{p^n}/\FF_p) \rto^-{\zeta \times 1} & \ZZ_p \times \Gal(\FF_{p^n}/\FF_p)
\rto^-{p_1} & \ZZ_p
}
\]
The homomorphism $\zeta$ defines a cohomology class also called $\zeta = \zeta_n \in H^1(\GG_n,\ZZ_p)$.
In a further abuse of notation we also write $\zeta$ for image of this cohomology class
in $H^1(\GG_n,E_0)$ induced by the inclusion $\ZZ_p \subseteq E_0$.
\end{rem}

There is a class $\zeta \in H^1(\GG_n,E_0)$ when $p=2$ as well; it takes a few more words to define. 

The following result
is implicit in various forms in various places. At $n=2$ and $p=3$ the exact sequence
of (\ref{fund-seq-exp}) is used in the work of Karamanov \cite{KaraPic}; see also  the work on the Picard
group in \cite{Lader} \S 5.

Recall that
$$
r(n) = \frac{p^n-1}{p-1} = p^{n-1} + \cdots + p + 1.
$$
Recall also that we are writing the group operation in $H^1(\GG_n,E_0)$ multiplicatively. In the following
result we will confuse a cohomology class with a crossed homomorphism representative. 

\begin{prop}\label{detred-alg} Let $p > 2$ and let $\zeta \in H^1(\GG_n,E_0)$ be the 
class defined in Remark \ref{define-zeta} . Then
$$
\mathrm{exp}(p\zeta) = t_0^{-\lambda}\det \in H^1(\GG_n,E_0^\times)
$$
where
$$
\lambda = \lim_k\ p^{nk}r(n) \in \lim_k\ \ZZ/p^{nk}(p^n-1) = \RC.
$$
\end{prop}

\begin{rem} We saw in Remark \ref{get-compl-straight} that the element $t_0^\lambda  \in
H^1(\GG_n,E_0^\times)$ is torsion of exact order $(p-1)$.
\end{rem}

\begin{proof} The class $\zeta$ is defined in Remark \ref{define-zeta} as a class in $H^1(\GG_n,\ZZ_p)$, so we examine
the diagram
\[
\xymatrix@C=40pt{
H^1(\GG_n,\ZZ_p) \rto^{\exp(p-)} \dto & H^1(\GG_n,\ZZ_p^\times) \dto\\
H^1(\GG_n,E_0) \rto^{\exp(p-)} & H^1(\GG_n,E^\times_0)\ .
}
\]
By construction $\exp(p\zeta) \in H^1(\GG_n,\ZZ_p^\times)$ is determined by the function on
$\GG_n \cong \SS_n \rtimes \Gal(\FF_{p^n}/\FF_p)$ given by
\[
(g,\phi) \mapsto g'(0)^{-r(n)}\det(g).
\]
Note that we have written $g'(0) \in \ZZ_p^\times$ for its image under the Teichm\"uller splitting $T:\FF_p^\times \to
\ZZ_p^\times$. From ( \ref{fund-alpha-det}) we have that
\[
t_0(g)^\alpha = g'(0)^\alpha = g'(0)\in E_0^\times .
\]
where $\alpha = \lim p^{nk} \in \RC$. Then for $(g,\phi) \in \SS_n \rtimes \Gal(\FF_{p^n}/\FF_p)$ we conclude
\[
t_0(g,\phi)^\lambda = ((t_0(g)^\alpha)^{r(n)} = g'(0)^{r(n)}.
\]
The result follows.
\end{proof}

Proposition \ref{detred-alg}, the long exact sequence of (\ref{fund-seq-exp}), and the identification of
cohomology classes with invertible modules yield a canonical isomorphism of graded Morava modules
$$
E_\ast (S^{2\lambda})/p \cong E_\ast[\det]/p
$$
and, thus, by arguments similar to those in Remark \ref{alg-pic-zero}, we have the following corollary. 

\begin{thm}\label{large-p-det} Let $2p > \max\{n^2+1,2n+2\}$ and 
$\lambda = \lim_k\ p^{nk}r(n)\in \RC$. Then there is a $K(n)$-local equivalence
$$
S^{2\lambda} \wedge V(0) \simeq S[\det] \wedge  V(0).
$$
\end{thm} 

\begin{rem}[{\bf The case $n=1$}]\label{nisone} The calculations at $n=1$ are somewhat 
anomalous because they are so simple, partly because $\GG_1 = \ZZ_p^\times$ and partly because
$\mm = (p) \subseteq E_0 = \ZZ_p$, so setting $p=0$ is a strong requirement.
But more than all that at $n=1$ we have $\det = t_0$ or, on invertible modules,
we have $E[\det] = E_{-2} = E_0S^2$. If $p > 2$, (\ref{fund-seq-exp}) becomes a short exact sequence in
cohomology
\[
\xymatrix@C=40pt{
0 \longr H^1(\GG_1,\ZZ_p) \rto^-{\exp(p-)} & H^1(\GG_1,\ZZ_p^\times)
\rto & H^1(\GG_1,\FF_p^\times) \longr 0
}
\]
is isomorphic to 
$$
0 \to \ZZ_p \to \ZZ_p^\times \to \FF_p^\times \to 0.
$$
The cohomology group $H^1(\GG_1,\ZZ_p^\times)$ is free of rank one over $\mathfrak{Z}_1 \cong \lim \ZZ/p^k(p-1)$
with generator $t_0$. The function $t_0(g)^\alpha=g'(0)$
is an element of order $p-1$, where we have confused $g'(0)$ with it Teichm\"uller lift to $\ZZ_p$. Finally, by
Proposition \ref{detred-alg}, the function 
$g \mapsto g'(0)^{-1}t_0(g)$ topologically generates the copy of the $p$-adic integers.
\end{rem}

For $n=2$ and primes $p > 2$, both the algebraic and the 
topological  Picard group have been calculated. The algebraic result is as follows. 
If $p >  3$, this result is proved explicitly
in \cite{Lader} \S 5, but it was known much earlier to the second author. At the prime 3,
the algebraic Picard group is the subject of \cite{KaraPic}.

\begin{prop}[{\bf The case $n=2$}]\label{pic-mod-p-n=2} Let $n=2$ and $p> 2$. The long exact sequence
$$
\xymatrix@C=40pt{
\cdots \to H^1(\GG_2, E_0) \rto^-{\mathrm{exp}(p-)} & H^1(\GG_2,E_0^\times) \rto &
H^1(\GG_2, (E_0/p)^\times)  \to \cdots
}
$$
becomes a short exact sequence
$$
\xymatrix{
0 \rto & \ZZ_p \rto & H^1(\GG_2,E_0^\times) \rto & 
\mathfrak{Z}_2\rto &0\ .
}
$$
The group $H^1(\GG_2,E_0)$ is topologically generated by $\zeta$ and the group
$H^1(\GG_2,E_0^\times)$ is topologically generated by cohomology classes defined by
$t_0$ and $\det$. Furthermore, the image of $\zeta$ in $H^1(\GG_2,E_0^\times)$ is
$$
t_0^{-\lambda}\det
$$
where $\lambda = \lim_k\ p^{2k}(p+1) \in  \mathfrak{Z}_2$.
\end{prop}

\begin{rem}\label{pic-mod-p-n=2-top} If $p > 2$, the map 
\[
e:\Pic^0_{K(2)} \longr H^1(\GG_2,E_0^\times)
\]
is a surjection as the target is topologically generated by $E_0S^2$ and $E_0S[\det]$. If $p >3$, it then
follows from Proposition \ref{why-the-hypothesis} that $e$ is an isomorphism. 
At the prime 3, the map $e$ is only onto and the kernel was computed in \cite{GHMRPic}. 
\end{rem}

\section{Two dualities}

In the next section -- see Theorem \ref{detred} -- we give a result relating
Brown-Comenetz and Spanier-Whitehead duality. We go through the two dualities in this section
and how they work with the Adams-Novikov Spectral Sequence (\ref{eq:ANSS}). The ring $R$ in this
section will be shorthand for $E_0$. 

\subsection{Spanier-Whitehead duality}

We define the $K(n)$-local Spanier-Whitehead dual of a spectrum $X$ by
$$
D_nX = F(X,L_{K(n)}S^0).
$$
A spectrum $X$ in the $K(n)$-local category is dualizable if for all $K(n)$-local spectra $Z$ the natural map
\[
D_nX \wedge Z \to F(X,Z)
\]
is a $K(n)$-local equivalence; that is, if $L_{K(n)}(D_nX \wedge Z) \simeq F(X,Z)$. By Theorem 8.6 of
\cite{HvStr} $X$ is dualizable if and only if $E_\ast X$ is finitely generated as an $E_\ast$-module. 

\begin{rem}\label{UCSS-rem} If $X$ is dualizable,  there is an isomorphism
\[
E_t D_nX = E^{-t}X.
\]
To compute $E^\ast X$ we use a Universal Coefficient Spectral Sequence.
The constructions is \S III.13 of \cite{AdamsBlue}, although an easy adaptation to the $K(n)$-local category is needed.
Combining this with the algebra of Appendix A of \cite{HvStr} we get a natural spectral sequence
\[
\Ext^p_{E_\ast}(E_\ast X,\Sigma^q E_\ast) \Longrightarrow E^{p+q}X
\]
where $\Ext_{E_\ast}$ denotes the derived functors of the functor $\Hom_{E_\ast}$ of continuous homomorphisms
in the category of graded $L_0^\mm$-complete $E_\ast$-modules. See Remark \ref{Morava-mods}. 
By definition $\Sigma^q E_\ast = \pi_\ast \Sigma^qE$. Since 
\[
\Hom_{E_\ast}(E_\ast X,\Sigma^q E_\ast) \cong \Hom_{E_0}(E_qX,E_0)
\]
we can rewrite the spectral sequence as
\begin{equation}\label{E-UCSS-note}
\Ext^p_{E_0}(E_qX,E_0)  \Longrightarrow E^{p+q}X. 
\end{equation}
The abutment has a continuous action of $\GG_n$ through the action on $E$; this spectral sequence becomes 
a $\GG_n$-equivariant spectral sequence if we give the $E_2$-term the conjugation action. Concretely,
if $\varphi \in \Hom_{E_0}(M,N)$ and $g \in \GG_n$ we have $(g\varphi)(x) = g\varphi(g^{-1}x)$. 
\end{rem} 

If $E_\ast X$ is $\mm$-torsion -- for
example if $X$ is a type-$n$ complex -- then (\ref{UCSS-simp}) and (\ref{E-UCSS-note}) imply that
$$
E_tD_nX\cong \Hom_R(E_{-t-n}X,R/\mm^\infty).
$$
This is an isomorphism of Morava modules if we use the conjugation action on the module of
homomorphisms. We then define
a functor $D_S(-)$ on  $\mm$-torsion Morava modules to Morava modules by
\[
D_S(M) = \Hom_{E_0}(M,R/\mm^\infty)
\]
with the conjugation action, so that we have
\begin{equation}\label{keyDS}
E_tD_nX\cong D_S(E_{-t-n}X).
\end{equation}

\begin{rem}[\textbf{SW duality and the ANSS spectral sequence}]\label{SWANSS} Suppose $E_\ast X$
is finite and, hence, $\mm$-torsion. The isomorphism of  Morava modules (\ref{keyDS})
allows us to rewrite the Adams-Novikov Spectral Sequence for $D_nX$ as
\begin{equation}\label{ANSSD}
H^s(\GG_n,D_S(E_{-t-n}X)) \Longrightarrow \pi_{t-s}D_nX.
\end{equation}
\end{rem}

\begin{exam}\label{nistwotypetwo} For self-dual $X$, the self-duality extends to the Adams-Novikov Spectral
Sequence. To be specific, we work at chromatic height $2$ and at $p>2$. Let $V(0)$ be the mod $p$ Moore
spectrum and let $v_1:\Sigma^{2(p-1)}V(0) \to V(0)$ be a  $v_1$-self map. We define
$M(1,n)$ by the  cofiber sequence 
\[
\xymatrix{
\Sigma^{2n(p-1)}V(0) \rto^-{v_1^n} & V(0) \rto & M(1,n).
}
\]
Since $v_1 = u_1u^{1-p} \in E_\ast V(0) = E_\ast/p$, we have that $E_\ast M(1,n)$ is concentrated in even degrees
and $E_0M(1,n) = \FF_{p^2}[u_1]/(u_1^{n})$

Because $p>2$, the $v_1$-self map of $V(0)$ is unique up to a unit in $\ZZ/p$, so we have a $K(n)$-equivalence
$D_2M(1,n) \simeq \Sigma^{-2n(p-1)-2}M(1,n)$. This fact and (\ref{keyDS}) give isomorphisms of Morava modules
\[
E_{2t+2n(p-1)+2}M(1,n) \cong E_{2t}D_2M(1,n) \cong \Hom_{E_0}(E_{2t-2}M(1,n),E_0/\mm^\infty)
\]
which sends the generator $u^{-t-n(p-1)}$ of $E_{2t+2n(p-1)+2}M(1,n)$ 
to the function
\[
u^{-t+1} \mapsto 1/pu_1^n = u^{-n(p-1)}/pv_1^n.
\]
Making these substitutions for $M(1,n)$, the Adams-Novikov Spectral Sequence (\ref{ANSSD}) becomes
$$
H^s(\GG_n,E_{2t+2n(p-1)+2}M(1,n)) \Longrightarrow \pi_{2t-s}\Sigma^{-2n(p-1)-2}
L_{K(2)}M(1,n).
$$
This is, up to reindexing, isomorphic to the original Adams-Novikov Spectral Sequence for $M(1,n)$.
\end{exam}

\subsection{Brown-Comentez duality}

Let $\ZZ/p^\infty = \QQ/\ZZ_{(p)}$ and let $I$ be the Brown-Comenetz
dual of the sphere; this spectrum is defined by the natural isomorphism
$$
\Hom(\pi_0 X, \ZZ/p^\infty) \cong [X,I].
$$
If $X$ is $K(n)$-local and 
if $M_nX$ is the fiber of $X \to L_{n-1}X=L_{K(0)\vee\ldots\vee K(n-1)}X$, then we define the
$K(n)$-local Brown-Comenetz dual\footnote{The nomenclature shifted some time after the appearance
of \cite{HvStr} and $K(n)$-local Brown-Comenetz duality became to be called Gross-Hopkins duality. We adhere to the older 
version.} of $X$ to be 
$$
I_nX = F(M_nX,I) \simeq F(X,I_n).
$$
Here $I_n = I_nL_{K(n)}S^0$.

By the work of Gross and the second author, we know the Morava module $E_0I_n$. Let's recall the outline of the argument
from \cite{StrickGrossHop}. The fact that $\GG_n$ is a virtual Poincar\'e duality 
group of dimension $n^2$ and the fact that the top local cohomology group
of $R = E_0$ is in degree $n$ (see (\ref{Rminfty})) yield an isomorphism of Morava
modules
$$
E_0 I_n= E_{-n^2-n} \otimes_R \Omega^{n-1}
$$
where
$$
\Omega^{n-1} = \bigwedge^{n-1} \Omega_{R/\WW} = R\{ du_1 \wedge \ldots \wedge du_{n-1}\}
$$
is the top exterior power of the differentials on $R=E_0$ over $\WW = W(\FF_{p^n})$.
It follows that $I_n$ is dualizable in the $K(n)$-local category, and hence
\begin{equation}\label{duals-togegther}
I_nX \simeq D_nX \wedge I_n.
\end{equation}
One of the main results of \cite{GHAlt} is that
there is an isomorphism of Morava modules
\begin{equation}\label{EINbis}
\Omega^{n-1} \cong E_{2n}[det]
\end{equation}
and hence
\begin{equation}\label{EIN}
E_0 I_n \cong E_{-n^2+n}[\det] \cong E_0S^{n^2-n}[\det].
\end{equation}
From this it follows that for $X$ finite and $\mm$-torsion
\begin{equation}\label{ebc} 
E_t I_nX \cong \Hom_R(E_{-t-2n+n^2}X,E_0[\det]/\mm^\infty)
\end{equation}

\begin{rem}[\textbf{BC duality and the ANSS}]\label{BCANSS} We continue to write $R=E_0$.
If $M$ is finite and $\mm$-torsion, define
$D_IM = \Hom_R(M,\Omega^{n-1}/\mm^\infty)$
with the conjugation action. Then we have from \ref{EINbis}) and (\ref{ebc}) that
\begin{equation}\label{keyDI}
E_tI_nX = D_I(E_{-t+n^2}X)
\end{equation} 
if $X$ is finite and $E_\ast X$ is $\mm$-torsion.
In \cite{StrickGrossHop}, Strickland defined an equivariant residue map
$$
H^{n^2}(\GG_n,\Omega^{n-1}/\mm^\infty) \to \ZZ/p^\infty.
$$
The residue map then defines a pairing
$$
H^s(\GG_n,M) \otimes H^{n^2-s}(\GG_n,D_IM) \to \ZZ/p^\infty\ .
$$
Under the hypothesis that $p-1 > n$ (compare Proposition \ref{e2-van-line}), this is a perfect pairing.
The Adams-Novikov Spectral Sequence for $X$ 
\begin{equation}\label{ANSSI}
H^s(\GG_n,D_I(E_{-t+n^2}X)) \Longrightarrow \pi_{t-s}I_nX.
\end{equation}
can then be rewritten as
\[
H^{n^2-s}(\GG_n,E_{-t+n^2}X)^\vee \Longrightarrow (\pi_{s-t}X)^\vee
\]
where $A^\vee = \Hom(A,\ZZ/p^\infty)$. This dualized spectral sequence is presumably dual to the original spectral sequence. 
\end{rem}

\section{The duality theorem}

Here we come to the algebraic and topological consequence of Proposition \ref{detred-alg}.
We have defined an element
$$
\lambda = \lim_k p^{nk}r(n) \in \lim_k \ZZ/p^{nk}(p^n-1) = \RC.
$$
Recall from Remark \ref{Morava-mods} that a graded Morava module $M_\ast$ is finite if $M_0 $ and $M_{1}$ are finite. 
Such a Morava module is automatically $\mm$-torsion, where $\mm \subseteq E_0$ is the maximal ideal.
The results of the previous sections now can be combined to give the following.

\begin{prop}\label{alg-DH} 1.) Let $M$ be a finite Morava module and
suppose $pM=0$. Then there is a natural isomorphism of Morava modules
$$
D_IM \cong E_0(S^{2\lambda - 2n}) \otimes_{E_0} D_SM.
$$

2.) Suppose $E_\ast X$ is a finite Morava module and
suppose $pE_\ast X=0$. Then there is a natural isomorphism
$$
E_\ast I_n X \cong E_{0}(S^{2\lambda + n^2-n}) \otimes_{E_0} E_\ast D_nX.
$$
\end{prop}

\begin{proof} This all follows from Proposition \ref{detred-alg}, which gives an isomorphism
of Morava modules
$$
E_0[\det]/p \cong E_0(S^{2\lambda})/p.
$$
For part (1), we combine the following isomorphisms
\begin{align*}
D_IM &\cong \Hom_R(M,\Omega^{n-1}/\mm^\infty)\\
&\cong \Hom_R(M,R/m^\infty) \otimes_R E_{2n}[det]\\
&\cong D_SM \otimes_R E_0(S^{2\lambda-2n}).
\end{align*}
We use Proposition \ref{detred-alg} in the last line. Part (2) follows from Part (1)
and the isomorphisms
\begin{align*}
E_t I_nX &\cong D_I(E_{{-t+n^2}}X)\\
E_t D_nX & \cong D_S(E_{-t-n}X)
\end{align*}
of (\ref{keyDS}) and (\ref{keyDI}). 
\end{proof}

\begin{prop}\label{top-DH} Let $2p > \mathrm{max}\{n^2+1,2n+2\}$ and let $X$ be a finite type $n$ spectrum 
with $p1_X=0: X\to X$. Then there is a weak equivalence in the $K(n)$-local category 
$$
I_n X \cong S^{2\lambda + n^2-n} \wedge D_nX.
$$
\end{prop}

\begin{proof} We have that
$$
I_nX \simeq D_n X\wedge I_n \simeq D_n  X \wedge S^{n^2-n}[\det].
$$
We also have 
$$
S^{n^2-n}[\det] \wedge V(0) \simeq S^{2\lambda + n^2 -n} \wedge V(0).
$$
Given our assumptions on $n$ and $p$, this is equivalent to the statement that
$$
E_0I_n/p = E_{-n+n^2}[\det]/p \cong E_{-n+n^2}(S^{2\lambda})/p.
$$
See Proposition \ref{why-the-hypothesis}. Since $pE_\ast X = 0$ we have natural isomorphisms of Morava modules
\begin{align}\label{lots-a-isos}
E_\ast I_nX &\cong E_\ast D_nX \otimes_{E_0} E_0S^{n^2-n}[\det]\nonumber\\
& \cong E_\ast D_nX \otimes_{E_0} E_0S^{2\lambda + n^2-n}\\
& \cong E_\ast (D_nX \wedge S^{2\lambda +n^2 -n})\nonumber.
\end{align}
which we are trying to extend to an equivalence of spectra. 

Because $p1_X = 0$ and $p>2$, we have a split cofiber sequence
$$
\xymatrix{
D_nX \rto^-i &D_nX \wedge V(0) \rto^-q & \Sigma D_nX\ .
} 
$$
We choose a section $s:\Sigma D_nX \to D_nX \wedge V(0)$ of $q$. Because the isomorphisms of (\ref{lots-a-isos})
are natural, we have a commutative diagram of split short exact sequences of Morava modules
\begin{equation}\label{big-one}
\xymatrix{
0 \dto & 0 \dto\\
E_\ast D_nX \wedge S^{n^2 -n}[\det]  \dto^{i_\ast} \rto^\cong & E_\ast D_nX \wedge S^{2\lambda +n^2 -n} \dto^{i_\ast}\\
E_\ast (D_nX \wedge S^{n^2 -n}[\det]  \wedge V(0)) \dto^{q_\ast} \rto^\cong &
E_\ast (D_nX \wedge S^{2\lambda +n^2 -n} \wedge V(0)) \dto^{q_\ast}\\
E_\ast \Sigma D_nX \wedge S^{n^2 -n}[\det]   \dto \rto^\cong &
E_\ast \Sigma D_nX \wedge S^{2\lambda +n^2 -n} \dto\\
0 & 0.
}
\end{equation}
We will be done if we can find a map $f:D_nX \wedge S^{n^2 -n}[\det] \to D_nX \wedge S^{2\lambda +n^2 -n}$
realizing the topmost horizontal map of Morava modules. By Theorem \ref{large-p-det} we have an equivalence
$S^{n^2 -n}[\det] \wedge V(0) \to S^{2\lambda +n^2 -n} \wedge V(0)$, so we do have a realization of the middle
horizontal map;
that is, we have a diagram in which we'd like to complete the dotted arrow:
$$
\xymatrix{
D_nX \wedge S^{n^2 -n}[\det]  \rto^-{i} \ar@{-->}[d]_f&
D_nX \wedge S^{n^2 -n}[\det]   \wedge V(0)\dto^g \\
D_nX \wedge S^{2\lambda +n^2 -n} \rto^-{i} &
D_nX \wedge S^{2\lambda +n^2 -n} \wedge V(0) \rto^-{q} 
&  D_nX \wedge S^{2\lambda +n^2 -n}.
}
$$
For the composite $qgi:D_nX \wedge S^{n^2 -n}[\det] \to D_nX \wedge S^{2\lambda +n^2 -n}$ we have
$$
E_\ast(qgi)=0
$$
as the algebraic diagram (\ref{big-one}) above commutes and the columns are exact sequences.
Using the section $s:\Sigma D_nX \to D_nX \wedge V(0)$ chosen above we form
$$
g' = g - sqg.
$$
Then $qg'i = qgi-qsqgi=0$ and
$$
E_\ast (g'i) = E_\ast (gi) - E_\ast(s)E_\ast(qgi) = E_\ast(gi).
$$
Thus, we can let $f$ be the unique map so that $if = g'i$. 
\end{proof} 

We finish by interpreting the number $\lambda$ as an integer in some special cases. We will say a
Morava module has a $v_n^{p^k}$-self map if $v_n^{p^k}=u^{p^k(1-p^n)}$ induces an isomorphism of Morava modules. 
$$
M \longr E_{2p^k(p^n-1)} \otimes_{E_0} M
$$

\begin{thm}\label{detred}  Let $M$ be finite  Morava module. If $\mm^kM = 0$, then $M$ has a $v_n^{p^k}$-self map.
If, in addition, $pM=0$ then
$$
D_IM = \Sigma^{2p^{nk}r(n)-2n} D_SM.
$$
\end{thm} 

\begin{proof} To prove the first statement, note that the class $v_n = u^{1-p^n} \in E_{2(p^n-1)}$ is $\GG_n$-invariant
modulo $\mm$. Since $\GG_n$ acts through graded ring maps, $v_n^{p^k}$ is invariant modulo $\mm^{k+1}$.
The second statement follows from Remark \ref{reduced-by-vn-alg} and Proposition \ref{alg-DH}. 
\end{proof}

\begin{thm}\label{detred-top} Let $2p > \mathrm{max}\{n^2+1,2n+2\}$ and let $X$ be a finite type $n$ spectrum with
a $v_n^{p^k}$-self map. Suppose $p1_X=0:X \to X$. Then
$$
I_nX = \Sigma^{2p^{nk}r(n)+n^2-n} D_nX.
$$
\end{thm} 

\begin{proof} This follows from Propositions \ref{reduced-by-vn} and \ref{top-DH}.
\end{proof}

\begin{rem}\label{self-dual} Theorem \ref{detred-top}  has strong implications for the Adams-Novikov spectral sequence
of a Spanier-Whitehead self-dual complex. For example, let $n=2$ and $p> 3$. Let $M(1,s)$ be the cofiber of
$v_1^s:\Sigma^{2s(p-1)}V(0) \to V(0)$; see Example \ref{nistwotypetwo}. Then, $M(1,s)$ is Spanier-Whitehead
self dual. We check $M(1,s)$ has a $v_2^{p^k}$-self map as long as $s \leq p^k$. Furthermore
$D_2M(1,s) = \Sigma^{-qn-2}M(1,s)$ where $q = 2(p-1)$. Then
$$
I_2M(1,p^k) = \Sigma^{2p^{2k}(p+1) +2} D_2M(1,p^k) = \Sigma^{2p^{2k}(p+1) - p^kq}
L_{K(2)}M(1,p^k).
$$
In combination with Remark \ref{SWANSS} and Remark \ref{BCANSS} this forces a very strong
duality onto the Adams-Novikov Spectral Sequence for $M(1,s)$. This is often what people mean
when they point to charts to discuss Gross-Hopkins duality.

The spectrum $M(1,1)$ is the Smith-Toda complex $V(1)$ and we have
\[
I_2V(1) \simeq \Sigma^{2(p+1) +2} D_2V(1) \simeq \Sigma^4 V(1).
\]
For this example, the duality in the Adams-Novikov Spectral Sequence discussed in Remark \ref{BCANSS} can be checked
by hand. In fact, $E_\ast V(1) = \FF_{p^2}[u^{\pm 1}]$ and there is an isomorphism
\[
\FF_p[v_2^{\pm 1}] \otimes \Lambda(\zeta_2) \otimes A \cong H^\ast(\GG_2,E_\ast V(1))
\]
where $A$ is the $3$-dimensional Poincar\'e duality algebra on classes
\begin{align*}
h_0 \in H^1(\GG_2,E_{2(p-1)}V(1))\qquad & h_1 \in H^1(\GG_2,E_{-2(p-1)}V(1))\\
g_0 \in H^2(\GG_2,E_{2(p-1)}V(1))\qquad& g_1 \in H^2(\GG_2,E_{-2(p-1)}V(1))\\
h_0g_1 = g_0h_1 &\in h_0 \in H^3(\GG_2,E_0V(1)).
\end{align*}
See Theorem 6.3.22 of \cite{RavVeryGreen}. 
\end{rem}

\begin{rem}\label{counter-exam} The hypothesis that $p$ be large with respect to $n$ is needed in Theorem \ref{detred}.
At the prime $2$ there is no non-trivial finite CW spectrum $X$ such that the identity $1_X: X \to X$ has order $2$.
To see this let $n$ be the largest integer so that there is a class $x \in H^\ast X = H^\ast(X,\FF_2)$ with $\Sq^n(x) \ne 0$.
Then if $a \in H^0V(0)$ is the bottom class, we have $\Sq^{n+1}(x \times a) \ne 0$ in $H^\ast(X \wedge V(0))$,
so $X \wedge V(0)$ cannot split as $X \vee \Sigma X$. This simple argument is due to Mark Mahowald.

The prime $2$ may always be an outlier, but at $n=2$ and $p=3$, the Smith-Toda complex $V(1)$ is a counterexample
to extending Theorem \ref{detred}. The identity map $V(1) \to V(1)$ does have order $3$.
By \cite{BehPem}, $V(1)$ has a $v_2^9$-self map, and hence
$L_{K(2)}V(1)$ is $144$-periodic. In \cite{GHMV1} we showed that the homotopy groups of $L_{K(2)}V(1)$ are
exactly 144-periodic; see also the charts at the end of \cite{bcdual}. 
In particular, there is no $v_2^3$-self map. We also noted in  \S 7 of \cite{bcdual}
that $I_2 \wedge V(1) \simeq \Sigma^{-22}V(1)$; see also \cite{BehrensQ}. This implies
$$
I_2V(1) \simeq I_2\wedge D_2V(1) \simeq \Sigma^{-28}L_{K(2)}V(1) \simeq \Sigma^{116}L_{K(2)}V(1)
$$
since $D_2V(1) \simeq \Sigma^{-6}L_{K(2)}V(1)$. 
By contrast
$$
\Sigma^{2p^{nk}r(n)+n^2-n} D_nV(1) \simeq \Sigma^{650}D_2V(1) \simeq \Sigma^{644} L_{K(2)}V(1).
$$
Since $644 \equiv 68$ modulo $144$, the equivalence of Theorem \ref{detred-top} can't hold. The map
$$
e:\Pic_{K(2)}^0 \longr H^1(\GG_2,E_0^\times)
$$
is not injective at $n=2$ and $p=3$. The difference
$116-68=48$ is related to the fact that there is an element $P$ in the kernel of the homomorphism of
$e$ (\ref{exotic-e})  so that $P \wedge V(1) \simeq
\Sigma^{48}L_{K(2)}V(1)$.  All of this is 
covered in excruciating detail in \cite{GHMRPic} and \cite{bcdual}. Scholars in this field will recognize
the number $48$ as having near-cabalistic significance; this is only one of the many places it appears.
\end{rem}

\bibliographystyle{alpha}

\bibliography{bibduality}

\end{document}